\documentclass{amsart}
\usepackage{amsfonts,amssymb,amsmath,euscript}
\usepackage{graphicx}
\usepackage[matrix,arrow,curve]{xy}
\usepackage{wrapfig}

  \usepackage{amssymb}
\usepackage{mathrsfs}
\usepackage{euscript}
\usepackage{lipsum}

    \newtheorem{thm}{Theorem}                   %  [section]
    \newtheorem{thm*}{Theorem}

    \newtheorem{lemma*}{Lemma}    %for AMS \newtheorem*

\newcommand*{\brs}[1]{\left(#1\right)}             % brackets
\newcommand{\mbR}{\mathbb R}
\newcommand{\mbC}{\mathbb C}

\newcommand{\Tr}{\operatorname{Tr}}
\newcommand{\rank}{\operatorname{rank}}

\begin{document}
\title{On a property of Herglotz functions}
% Former title: 
%\title{On analytic extension of resonance points}
\author{Nurulla Azamov}
\address{Independent scholar, Adelaide, SA, Australia}
\email{azamovnurulla@gmail.com}
\begin{abstract} 
In this note I prove the following property of Herglotz functions, which to my knowledge is new:
For a Herglotz function $h(z)$ and a real number $r \in \mathbb R$ define a Herglotz function $g_r(z) = (r - h(z))^{-1}.$ 
Let $\mu_r^{(s)}$ be the singular part of the measure $\mu_r$ which corresponds to $g_r(z)$ via the Herglotz representation theorem.
Then the measure
$
   \int_0^1 \mu_r^{(s)}\,dr
$
is absolutely continuous, its density is integer-valued a.e., and moreover the density takes values $0$ or $1$ a.e. 

\end{abstract}
\maketitle
A Herglotz function, also known as a Nevanlinna, Pick or $R$-function, is a holomorphic map of the upper complex half-plane, $\mbC_+,$  into itself, named after \cite{He}. Herglotz functions and their operator-valued analogues 
are an important tool in spectral and scattering theories, see e.g. \cite{AiWa,AD56,ChP,Do,GTs,SiTrId2,Teschl,YaBook}. 
%They were also used in the author's work \cite{Az3v6,Az9,AzDa} on singular spectral shift function (SSSF). 
During a work related to the singular spectral shift function (SSSF) I stumbled upon a property of Herglotz functions which, to the best of my knowledge, is new and seems to be interesting.
%, even though, given importance of Herglotz functions in a variety of branches of mathematics and physics, it is probably safe to say that any novel property of Hergltoz functions is interesting. 
Herglotz functions have numerous applications, but here I use SSSF to prove something about Herglotz functions. The proof is short but that is at the expense of using a property of SSSF which has a not-so-short proof and in this regard it would be interesting to find a direct proof. 

For any Herglotz function $h(z)$ there exist a non-negative real number $\alpha,$ a real number $\beta_i$ and a positive Borel measure $\mu$ on $\mbR$
obeying 
  $\int \frac{1}{1+\lambda^2}\,d\mu(\lambda) < \infty,$
such that the \emph{Herglotz representation formula} 
$$
  h(z) = \alpha z + \beta_i + \int_\mbR \brs{\frac{1}{\lambda-z} - \frac {\lambda}{\lambda^2+1}}\,d\mu(\lambda),
$$
holds, see e.g. \cite{AD56}. 
Vice versa, for a non-negative real number $\alpha,$ a real number $\beta_i,$ and a Herglotz measure $\mu$ 
the function~$h(z)$ is Herglotz. So, there is one-one correspondence between Herglotz functions and triples $[\alpha,\beta_i,\mu].$
%The formula (\ref{F: integral repr-n for Herglotz f-ns}) is in fact Herglotz representation relative $i,$ and so the notation $\beta_i;$
%similar representations can be written relative any $a$ from $\mbC_+,$ see \cite{AD56}.

If $h(z)$ is a Herglotz function, then for any real number $r$ the function
    $ g_r(z) :=  (r-h(z))^{-1}$
is also Herglotz. 
Let $\mu_r$ be the measure which corresponds to $g_r(z),$ and let 
$\mu_r^{(s)}$ be the singular part of $\mu_r.$ Consider the averaged measure:
\begin{equation} \label{F: integral of mu[r,(s)]}
  \int_0^1 \mu_r^{(s)}\,dr.
\end{equation}

%: {T: averaged sing meas of Herglotz functions}
\begin{thm} \label{T: averaged sing meas of Herglotz functions}
The averaged measure (\ref{F: integral of mu[r,(s)]})
is absolutely continuous and its density is a function which takes values $0$ or $1$ a.e. 
\end{thm}
\begin{proof}
It is a well-known property of a Herglotz function $f(z)$ (see e.g. \cite{Do,SiTrId2})
that for some self-adjoint operator~$H_0$ with simple spectrum and a rank 1 
non-negative self-adjoint operator~$V$ we have $f(z) = \Tr(R_z(H_0)V),$ where $R_z(H_0) = (H_0 - z)^{-1}$ is the resolvent of  $H_0.$
Applying this to the Herglotz function $-h(z)^{-1}$
gives $-h(z)^{-1} = \Tr(R_z(H_0)V).$ From this a simple well-known calculation implies 
$$
  (r - h(z))^{-1} = \Tr(R_z(H_r)V),
$$
where $H_r = H_0 + rV.$
It follows that the measure $\mu_r^{(s)}$ is the singular part of the measure 
$\Delta \mapsto \Tr(E_\Delta(H_r)V).$ Thus, the measure
(\ref{F: integral of mu[r,(s)]}) is the singular spectral shift measure 
of the pair $H_0,V.$ Hence, 
by \cite[Theorem 8.2.6]{Az3v6} (see also \cite{AzDaMN} for a shorter proof) it is absolutely continuous and its density is a.e. integer-valued.
Since $V \geq 0$ and $\rank V = 1$ the density takes values $0$ or $1$ a.e. 
\end{proof}

\bigskip 
{\it Acknowledgements.} I thank my wife for financially supporting me during the work on this paper.


\begin{thebibliography}{XXXXX}

\bibitem{AiWa}  M.\,Aizenman, S.\,Warzel,  {\it Random Operators: disorder effects on Quantum Spectra and Dynamics,} Grad. Stud. Math. (Amer. Math. Soc., 2015)

\bibitem{AD56}   N.\,Aronszajn, W.\,F.\,Donoghue, {\it  On exponential representation of analytic
        functions in the upper half-plane with positive imaginary part,}   J.\,d'Anal. Math. {\bf 5},  321-388 (1956) 

\bibitem{Az3v6} N.\,A.\,Azamov, {\it Absolutely continuous and singular spectral shift functions}, Dissertationes Math. {\bf 480}, 1-102  (2011)

%\bibitem{AzSFIES}     N.\,A.\,Azamov,       {\it Spectral flow inside essential spectrum,}      
%Dissertationes Math. {\bf 518}, 1-156 (2016)


\bibitem{AzDaMN}     N.\,A.\,Azamov and T.\,W.\, Daniels,      {\it Singular spectral shift function for resolvent comparable operators,}
Math. Nachr. {\bf 292}, 1911-1930 (2019)       

\bibitem{ChP}  Y.\,T.\,Christodoulides, D.\,B.\,Pearson, {\it  Spectral theory of Herglotz functions 
            and their compositions,}  Math. Physics, Analysis and Geometry,  {\bf 7},  333--345  (2004)

\bibitem{Do}     W.\,F.\,Donoghue,     {\it Monotone Matrix Functions and Analytic Continuation,}
Springer, Berlin, Heidelberg, New York, 1974

\bibitem{GTs} F.\,Gesztesy, E.\,Tsekanovkii, {\it  On matrix-valued Herglotz functions,}  Math. Nachr. {\bf 218}, 61-138 (2000)

\bibitem{He} G.\,Herglotz, {\it \"Uber Potenzreihen mit positivem, reellem Teil im Einheitskreis,} S\"achs. Acad. Wiss. Leipzig, {\bf 63}, 501-511  (1911)  

\bibitem{SiTrId2}     B.\,Simon,     {\it Trace Ideals and their Applications,} Second edition, Math. Surveys Monogr. (Amer. Math. Soc., 2005)

\bibitem{Teschl} G.\,Teschl, {\it Mathematical methods in quantum mechanics,} AMS, Graduate Studies in Mathematics, vol.\,157

\bibitem{YaBook}     D.\,R.\,Yafaev,     {\it Mathematical scattering theory: general theory,}. Providence, R.\,I., AMS, 1992

\end{thebibliography}
\end{document}